\documentclass[12pt,twoside]{amsart}
\usepackage{amssymb,amsmath,amsthm, amscd, enumerate, mathrsfs}
\usepackage{graphicx, hhline}
\usepackage[all]{xy}
\title[On log canonical rings]
{On the log canonical ring of projective plt pairs 
with the Kodaira dimension two}
\author{Osamu Fujino and Haidong Liu}
\date{2019/4/9, version 0.13}
\address{Department of Mathematics, Graduate School of Science, 
Osaka University, Toyonaka, Osaka 560-0043, Japan}
\email{fujino@math.sci.osaka-u.ac.jp}
\address{Department of Mathematics, Graduate School of Science, 
Kyoto University, Kyoto 606-8502, Japan}
\email{jiuguiaqi@gmail.com}

\DeclareMathOperator{\Supp}{Supp}

\newtheorem{thm}{Theorem}[section]
\newtheorem{lem}[thm]{Lemma}

\newtheorem{conj}[thm]{Conjecture}
\newtheorem{cor}[thm]{Corollary}

\theoremstyle{definition}

\newtheorem{rem}[thm]{Remark}
\newtheorem*{ack}{Acknowledgments}

\makeatletter
    
    \@addtoreset{equation}{section}
\makeatother

\begin{document}

\begin{abstract}
The log canonical ring of a projective 
plt pair with the Kodaira dimension two 
is finitely generated. 
\end{abstract}

\maketitle 
\tableofcontents

\section{Introduction}\label{g-sec1}
One of the most important open problems in the theory of 
minimal models for higher-dimensional algebraic varieties 
is the finite generation of 
log canonical rings for lc pairs. 

\begin{conj}\label{g-conj1.1}
Let $(X, \Delta)$ be a projective 
lc pair defined over $\mathbb C$ such that 
$\Delta$ is a $\mathbb Q$-divisor on $X$. 
Then the log canonical ring 
$$
R(X, \Delta):=\bigoplus _{m\geq 0}H^0(X, \mathcal 
O_X(\lfloor m(K_X+\Delta)\rfloor))
$$ is a finitely generated $\mathbb C$-algebra. 
\end{conj}

In \cite{fujino-gongyo2}, Yoshinori Gongyo and the first author 
showed that Conjecture \ref{g-conj1.1} is closely related to the 
abundance conjecture and is essentially equivalent 
to the existence problem of good minimal models for 
lower-dimensional varieties. 
Therefore, Conjecture \ref{g-conj1.1} is thought to be a very 
difficult open problem. 

When $(X, \Delta)$ is klt, Shigefumi Mori and the first author 
showed that 
it is sufficient to prove Conjecture \ref{g-conj1.1} under the 
extra assumption that $K_X+\Delta$ is big in \cite{fujino-mori}. 
Then Birkar--Cascini--Hacon--M\textsuperscript{c}Kernan completely 
solved Conjecture \ref{g-conj1.1} for projective klt 
pairs in \cite{bchm}. More generally, in \cite{fujino-some}, 
the first author slightly generalized a canonical bundle formula in 
\cite{fujino-mori} and showed that 
Conjecture \ref{g-conj1.1} holds true even 
when $X$ is in Fujiki's class $\mathcal C$ 
and $(X, \Delta)$ is klt. 
We note that Conjecture \ref{g-conj1.1} is not necessarily true when 
$X$ is not in Fujiki's class $\mathcal C$ (see \cite{fujino-some} 
for the details). 
Anyway, we have already established the finite generation of log canonical 
rings for klt pairs. So we are mainly interested in 
Conjecture \ref{g-conj1.1} for $(X, \Delta)$ 
which is lc but is not klt. 

If $(X, \Delta)$ is lc, then we have already known that 
Conjecture \ref{g-conj1.1} holds true when $\dim X\leq 4$ (see \cite{fujino-finite}). 
If $(X, \Delta)$ is lc and $\dim X=5$, 
then Kenta Hashizume showed that Conjecture \ref{g-conj1.1} holds true 
when $\kappa (X, K_X+\Delta)<5$ in \cite{hashizume}. We note that 
$$
\bigoplus _{m\geq 0} H^0(X, \mathcal O_X(\lfloor mD\rfloor))
$$ 
is always a finitely generated $\mathbb C$-algebra when $X$ is 
a normal projective variety and $D$ is a $\mathbb Q$-Cartier 
$\mathbb Q$-divisor on $X$ with $\kappa (X, D)\leq 1$. 
Therefore, the following theorem is the 
first nontrivial step towards the complete 
solution of Conjecture \ref{g-conj1.1} for higher-dimensional algebraic varieties. 

\begin{thm}[Main Theorem]\label{g-thm1.2}
Let $(X, \Delta)$ be a projective plt pair such 
that $\Delta$ is a $\mathbb Q$-divisor. 
Assume that $\kappa (X, K_X+\Delta)=2$. 
Then the log canonical ring 
$$
R(X, \Delta)=\bigoplus _{m\geq 0}H^0(X, \mathcal 
O_X(\lfloor m(K_X+\Delta)\rfloor))
$$ 
is a finitely generated $\mathbb C$-algebra. 
\end{thm}

In this paper, we will describe the proof of Theorem \ref{g-thm1.2}. 

\begin{ack}
The first author was partially 
supported by JSPS KAKENHI Grant Numbers 
JP16H03925, JP16H06337. 
The authors would like to thank Kenta Hashizume for useful 
discussions. A question he asked is one of the motivations of this paper. 
Finally they thank the referee for useful comments. 
\end{ack}

We will work over $\mathbb C$, the complex number field, throughout 
this paper. We will freely use the basic 
notation of the minimal model program as in 
\cite{fujino-fundamental} and \cite{fujino-foundations}. 
In this paper, we do not use $\mathbb R$-divisors. 
We only use $\mathbb Q$-divisors. 

\section{$\mathbb Q$-divisors} 
Let $D$ be a $\mathbb Q$-divisor on a normal variety $X$, that is, 
$D$ is a finite formal sum $\sum _i d_i D_i$ where 
$d_i$ is a rational number and $D_i$ is a prime divisor 
on $X$ for every $i$ such that $D_i\ne D_j$ for $i\ne j$. 
We put 
$$
D^{<1}=\sum _{d_i<1}d_iD_i, \quad 
D^{\leq 1}=\sum _{d_i\leq 1}d_i D_i, \quad 
\text{and} \quad D^{=1}=\sum _{d_i=1}D_i. 
$$
We also put 
$$
\lceil D\rceil =\sum _i \lceil d_i \rceil D_i, \quad 
\lfloor D\rfloor=-\lceil -D\rceil, \quad \text{and} 
\quad 
\{D\}=D-\lfloor D\rfloor, 
$$
where $\lceil d_i\rceil$ is the integer defined by 
$d_i\leq \lceil d_i\rceil <d_i+1$. A $\mathbb Q$-divisor 
$D$ on a normal variety $X$ is called a boundary $\mathbb Q$-divisor 
if $D$ is effective and $D=D^{\leq 1}$ holds. 

Let $B_1$ and $B_2$ be two $\mathbb Q$-divisors 
on a normal variety $X$. Then 
we write $B_1\sim _{\mathbb Q} B_2$ if there exists a positive integer $m$ 
such that $mB_1\sim mB_2$, that is, 
$mB_1$ is linearly equivalent to $mB_2$. 

Let $f:X\to Y$ be a proper surjective morphism 
between normal varieties and let $D$ be a $\mathbb Q$-Cartier 
$\mathbb Q$-divisor on $X$. 
Then we write $D\sim _{\mathbb Q, f}0$ if there exists a $\mathbb Q$-Cartier 
$\mathbb Q$-divisor $B$ on $Y$ such that $D\sim _{\mathbb Q}f^*B$. 

\bigskip 

Let $D$ be a $\mathbb Q$-Cartier $\mathbb Q$-divisor 
on a normal projective variety $X$. Let $m_0$ be a positive integer 
such that $m_0D$ is a Cartier divisor. 
Let $$\Phi_{|mm_0D|}: X\dashrightarrow \mathbb P^{\dim\!|mm_0D|}$$ be 
the rational map given by the complete linear system $|mm_0D|$ 
for a positive integer $m$. 
We put 
$$\kappa (X, D):=\max_m \dim \Phi_{|mm_0D|}(X)$$ if 
$|mm_0D|\ne \emptyset$ for some $m$ and 
$\kappa (X, D)=-\infty$ otherwise. 
We call $\kappa(X, D)$ the Iitaka dimension 
of $D$. 
Note that $\Phi_{|mm_0D|}(X)$ denotes the 
closure of the image of the rational map $\Phi_{|mm_0D|}$. 

Let $D$ be a $\mathbb Q$-Cartier $\mathbb Q$-divisor on a normal projective 
variety $X$. If $D\cdot C\geq 0$ for every curve $C$ on $X$, 
then we say that $D$ is nef. 
If $\kappa(X, D)=\dim X$ holds, then we say that $D$ is big. 

\medskip 

In this paper, we will repeatedly use the following 
well-known easy lemma. 

\begin{lem}\label{g-lem2.1}
Let $\varphi: X\to X'$ be a birational 
morphism between normal projective surfaces and let $M$ be 
a nef $\mathbb Q$-divisor 
on $X$. 
Assume that $M':=\varphi_*M$ is $\mathbb Q$-Cartier. 
Then $M'$ is nef. 
\end{lem}
\begin{proof}
By the negativity lemma, we can write 
$\varphi^*M'=M+E$ for some 
effective $\varphi$-exceptional $\mathbb Q$-divisor $E$ 
on $X$. 
We can easily see that $(M+E)\cdot C\geq 0$ for every  
curve $C$ on $X$. 
Therefore, $M'$ is a nef 
$\mathbb Q$-divisor on $X'$. 
\end{proof}

\section{Singularities of pairs}

Let us quickly recall the notion of singularities of pairs. 
For the details, we recommend the 
reader to see \cite{fujino-fundamental} and \cite{fujino-foundations}. 

\medskip 

A pair $(X, \Delta)$ consists of a normal 
variety $X$ and a $\mathbb Q$-divisor 
$\Delta$ on $X$ such that $K_X+\Delta$ is $\mathbb Q$-Cartier. 
Let $f:Y\to X$ be a projective 
birational morphism from a normal variety $Y$. 
Then we can write 
$$
K_Y=f^*(K_X+\Delta)+\sum _E a(E, X, \Delta)E
$$ 
with 
$$f_*\left(\underset{E}\sum a(E, X, \Delta)E\right)=-\Delta, 
$$ 
where $E$ runs over prime divisors on $Y$. 
We call $a(E, X, \Delta)$ the discrepancy of $E$ with 
respect to $(X, \Delta)$. 
Note that we can define the discrepancy $a(E, X, \Delta)$ for 
any prime divisor $E$ over $X$ by taking a suitable 
resolution of singularities of $X$. 
If $a(E, X, \Delta)\geq -1$ (resp.~$>-1$) for 
every prime divisor $E$ over $X$, 
then $(X, \Delta)$ is called sub lc (resp.~sub klt). 
If $a(E, X, \Delta)>-1$ holds for every exceptional divisor $E$ over 
$X$, then $(X, \Delta)$ is called sub plt.  
It is well known that $(X, \Delta)$ is sub lc if it is sub plt. 

\medskip 

Let $(X, \Delta)$ be a sub lc pair. 
If there exist a projective birational morphism 
$f:Y\to X$ from a normal variety $Y$ and a prime divisor $E$ on $Y$ 
with $a(E, X, \Delta)=-1$, then $f(E)$ is called an lc center of 
$(X, \Delta)$. 
We say that $W$ is an lc stratum of $(X, \Delta)$ when $W$ is an lc 
center of $(X, \Delta)$ or $W=X$. 

\medskip 

We assume that $\Delta$ is effective. 
Then $(X, \Delta)$ is 
called lc, plt, and klt if it is sub lc, sub plt, and sub klt, respectively. 
In this paper, we call $\kappa(X, K_X+\Delta)$ the Kodaira 
dimension of $(X, \Delta)$ when $(X, \Delta)$ is a 
projective lc pair. 

\section{Preliminary lemmas}\label{g-sec4}

In this section, we prepare two useful lemmas. 
One of them is a kind of connectedness lemma and will play 
a crucial role in this paper. 
Another one is a well-known generalization of the Kawamata--Shokurov 
basepoint-free theorem, which is essentially due to Yujiro Kawamata. 
We state it explicitly for the reader's convenience. 

\medskip 

The following lemma is a key observation. 
As we mentioned above, it is a kind of connectedness lemma and 
will play a crucial role in this paper. 

\begin{lem}\label{g-lem4.1}
Let $f:V\to W$ be a surjective morphism 
from a smooth projective variety $V$ onto 
a normal projective variety $W$. 
Let $B_V$ be a $\mathbb Q$-divisor 
on $V$ such that 
$K_V+B_V\sim _{\mathbb Q, f}0$, 
$(V, B_V)$ is sub plt, and $\Supp B_V$ is a simple 
normal crossing divisor. 
Assume that 
the natural map 
$$
\mathcal O_W\to f_*\mathcal O_V(\lceil -(B^{<1}_V)\rceil)
$$
is an isomorphism. 
Let $S_i$ be an irreducible component of $B^{=1}_V$ such that 
$f(S_i)\subsetneq W$ for $i=1, 2$. 
We assume that $f(S_1)\cap f(S_2)\ne \emptyset$. 
Then $S_1=S_2$ holds. 
In particular, we have $f(S_1)=f(S_2)$. 
\end{lem}

\begin{proof}
We note that $B^{=1}_V$ is a disjoint union of smooth prime divisors since 
$(V, B_V)$ is sub plt and $\Supp B_V$ is a simple normal crossing 
divisor. 
We put $C_i=f(S_i)$ for $i=1, 2$. 
Then we put $Z=C_1\cup C_2$ with the reduced 
scheme structure. 
By taking some suitable birational modification of $V$ and replacing 
$S_i$ with its strict transform for $i=1, 2$, 
we may further assume that 
$f^{-1}(Z)$ is a divisor and that 
$f^{-1}(Z)\cup \Supp B_V$ is contained in a simple normal crossing divisor. 
Let $T$ be the union of the irreducible 
components of $B^{=1}_V$ that are mapped into $Z$ by $f$. 
Let us consider the following short exact sequence 
$$
0\to \mathcal O_V(A-T)\to \mathcal O_V(A)\to \mathcal O_T(A|_T)\to 0
$$ 
with $A=\lceil -(B^{<1}_V)\rceil$. 
Then we obtain the long exact sequence 
\begin{equation*}
0 \longrightarrow 
f_*\mathcal O_V(A-T)\longrightarrow f_*\mathcal O_V(A)\longrightarrow 
f_*\mathcal O_T(A|_T)
\overset{\delta}\longrightarrow R^1f_*\mathcal O_V(A-T)\longrightarrow \cdots.  
\end{equation*}
Note that 
$$
A-T-(K_V+\{B_V\}+B^{=1}_V-T)=-(K_V+B_V)\sim _{\mathbb Q, f}0. 
$$ 
Therefore, by \cite[Theorem 6.3 (i)]{fujino-fundamental}, 
every nonzero local section of $R^1f_*\mathcal O_V(A-T)$ contains 
in its support the $f$-image 
of some lc stratum of $(V, \{B_V\}+B^{=1}_V-T)$. 
On the other hand, the support of $f_*\mathcal O_T(A|_T)$ is 
contained in $Z=f(T)$. 
We note that no lc strata of $(V, \{B_V\}+B^{=1}_V-T)$ are mapped 
into $Z$ by $f$ by construction. 
Therefore, the connecting homomorphism $\delta$ is a zero map. 
Thus we get a short exact sequence 
$$
0\to f_*\mathcal O_V(A-T)\to \mathcal O_W\to f_*\mathcal O_T(A|_T)\to 0. 
$$ 
Since $f_*\mathcal O_V(A-T)$ is contained in $\mathcal O_W$ and 
$f(T)=Z$, 
we have $f_*\mathcal O_V(A-T)=\mathcal I_Z$, where $\mathcal I_Z$ 
is the defining ideal sheaf of $Z$ on $W$. 
Thus, by the above short exact sequence, 
we obtain that the natural 
map $\mathcal O_Z\to f_*\mathcal O_V(A|_T)$ is an isomorphism. 
Hence we obtain 
$$
\xymatrix{
\mathcal O_Z\ar[r]^-{\sim}& f_*\mathcal O_T \ar[r]
^-{\sim}&f_*\mathcal O_T(A|_T). 
}
$$
In particular, $f:T\to Z$ has connected fibers. 
Therefore, $f^{-1}(P)\cap T$ is connected for every $P\in 
C_1\cap C_2$. 
Note that $T$ is a disjoint union of smooth prime divisors 
since $T\leq B^{=1}_V$. 
Thus we get $T=S_1=S_2$ since $S_1, S_2\leq T$. 
\end{proof}

As a corollary of Lemma \ref{g-lem4.1}, 
we have: 

\begin{cor}\label{g-cor4.2}
Let $f:V\to W$ be a surjective morphism 
from a smooth projective variety $V$ onto 
a normal projective variety $W$. 
Let $B_V$ be a $\mathbb Q$-divisor 
on $V$ such that 
$K_V+B_V\sim _{\mathbb Q, f}0$, 
$(V, B_V)$ is sub plt, and $\Supp B_V$ is a simple 
normal crossing divisor. 
Assume that 
the natural map 
$$
\mathcal O_W\to f_*\mathcal O_V(\lceil -(B^{<1}_V)\rceil)
$$
is an isomorphism. 
Let $S$ be an irreducible component of $B^{=1}_V$ such that 
$Z:=f(S)\subsetneq W$. 
We put $K_S+B_S=(K_V+B_V)|_S$ by adjunction. 
Then $(S, B_S)$ is sub klt and the natural map 
$$
\mathcal O_Z\to g_*\mathcal O_S(\lceil -B_S\rceil)
$$ 
is an isomorphism, where $g:=f|_S$. 
In particular, $Z$ is normal. 
\end{cor}

\begin{proof}
We can easily check that $(S, B_S)$ is sub klt by adjuction. 
We consider the following short exact sequence 
$$
0\to \mathcal O_V(\lceil -(B^{<1}_V)\rceil -S)
\to \mathcal O_V(\lceil -(B^{<1}_V)\rceil) \to \mathcal O_S
(\lceil -B_S\rceil)\to 0. 
$$ 
Note that $B^{<1}_V|_S=B^{<1}_S=B_S$ holds. 
By Lemma \ref{g-lem4.1}, 
we know that no lc strata of $(V, \{B_V\}+B^{=1}_V-S)$ are mapped 
into $Z$ by $f$. 
By the same argument as in the proof of Lemma \ref{g-lem4.1}, 
we obtain that the natural map 
$$
\mathcal O_Z\to g_*\mathcal O_S(\lceil -B_S\rceil) 
$$ 
is an isomorphism. 
Therefore, the natural map $\mathcal O_Z\to g_*\mathcal O_S$ 
is an isomorphism. 
This implies that $Z$ is normal. 
\end{proof}

Lemma \ref{g-lem4.3} is well known to the experts. It is a 
slight refinement of 
the Kawamata--Shokurov basepoint-free theorem and 
is essentially due to Yujiro Kawamata (see \cite[Lemma 3]{kawamata}). 

\begin{lem}\label{g-lem4.3}
Let $(V, B_V)$ be a projective plt pair and let $D$ be a nef Cartier 
divisor on $V$. 
Assume that 
$aD-(K_V+B_V)$ is nef and big for some $a>0$ and 
that $\mathcal O_V(D)|_{\lfloor B_V\rfloor}$ is semi-ample. 
Then $D$ is semi-ample. 
\end{lem}
\begin{proof}
By replacing $D$ with a multiple, 
we may assume that $\left|\mathcal O_V(mD)|_{\lfloor B_V\rfloor|}\right|$ 
is free for every nonnegative integer $m$. 
Since $(V, B_V)$ is plt, 
the non-klt locus of $(V, B_V)$ 
is $\lfloor B_V\rfloor$. 
Therefore, by \cite[Corollary 4.5.6]{fujino-foundations}, $D$ 
is semi-ample. 
\end{proof}

\section{On lc-trivial fibrations}\label{g-sec5}

In this section, we recall some results 
on klt-trivial fibrations in \cite{ambro-shokurov} and 
lc-trivial fibrations in \cite{fujino-gongyo} for the reader's convenience. 
We give only the definition which will be necessary 
to our purposes. 

\medskip 

Let $f:V\to W$ be a surjective morphism from a smooth projective 
variety $V$ onto a normal projective variety $W$. 
Let $B_V$ be a $\mathbb Q$-divisor on $V$ such that 
$(V, B_V)$ is sub lc and $\Supp B_V$ is a simple normal crossing divisor on $V$. 
Let $P$ be a prime divisor on $W$. 
By shrinking $W$ around the generic point of $P$, 
we assume that $P$ is Cartier. 
We set 
$$
b_P:=\max \left\{t \in \mathbb Q\, |\, 
 {\text{$(V, B_V+tf^*P)$ is sub lc over the generic point of $P$}} \right\}.  
$$ 
Then we put 
$$
B_W:=\sum _P (1-b_P)P, 
$$ 
where $P$ runs over prime divisors on $W$. 
It is easy to see that $B_W$ is a well-defined 
$\mathbb Q$-divisor on $W$ (see the proof of 
Lemma \ref{g-lem5.1} below). 
We call $B_W$ the discriminant $\mathbb Q$-divisor 
of $f:(V, B_V)\to W$. 
We assume that the natural map 
$$
\mathcal O_W\to f_*\mathcal O_V(\lceil -(B^{<1}_V)\rceil)
$$ 
is an isomorphism. 
In this situation, we have: 

\begin{lem}\label{g-lem5.1} 
$B_W$ is a boundary $\mathbb Q$-divisor on $W$. 
\end{lem}

We give a detailed proof of Lemma \ref{g-lem5.1} for the 
reader's convenience. 

\begin{proof}[Proof of Lemma \ref{g-lem5.1}]
We can easily see that there exists a nonempty Zariski open set 
$U$ of $W$ such that $b_P=1$ holds for every prime divisor 
$P$ on $W$ with $P\cap U\ne \emptyset$. 
Therefore, $B_W$ is a well-defined $\mathbb Q$-divisor on $W$. 
Since $(V, B_V)$ is sub lc, 
$b_P\geq 0$ holds for every 
prime divisor $P$ on $W$. 
Thus, we have $B_W=B^{\leq 1}_W$ by definition. 
If $b_P>1$ holds for some prime divisor $P$ on $W$, then we see that 
the natural map $\mathcal O_W\to f_*\mathcal O_V(\lceil -(B^{<1}_V)\rceil)$ 
factors through $\mathcal O_W(P)$ in a neighborhood of 
the generic point of $P$. 
This is a contradiction. 
Therefore, $b_P\leq 1$ always holds for 
every prime divisor $P$ on $W$. 
This means that $B_W$ is effective. 
Hence we see that $B_W$ is a boundary $\mathbb Q$-divisor 
on $W$.  
\end{proof}

We further assume that 
$K_V+B_V\sim _{\mathbb Q} f^*D$ for some $\mathbb Q$-Cartier 
$\mathbb Q$-divisor $D$ on $W$. 
We set $$M_W:=D-K_W-B_W, $$where 
$K_W$ is the canonical divisor of $W$. 
We call $M_W$ the moduli $\mathbb Q$-divisor 
of $K_V+B_V\sim _{\mathbb Q} f^*D$. 
Then we have: 

\begin{thm}\label{g-thm5.2}
There exist a birational morphism $p:W'\to W$ from a smooth projective 
variety $W'$ and a nef $\mathbb Q$-divisor 
$M_{W'}$ on $W'$ such that $p_*M_{W'}=M_W$. 
\end{thm}

Theorem \ref{g-thm5.2} is a special case of 
\cite[Theorem 3.6]{fujino-gongyo}, which is a generalization of 
\cite[Theorem 2.7]{ambro-shokurov}. When $W$ is a curve, we have: 

\begin{thm}[{\cite[Theorem 0.1]{ambro-shokurov}}]\label{g-thm5.3}
If $\dim W=1$ and $(V, B_V)$ is sub klt, 
then $M_W$ is semi-ample.  
\end{thm}

As an easy consequence of Theorem \ref{g-thm5.3}, we have: 

\begin{cor}\label{g-cor5.4}
If $\dim W=1$, $(V, B_V)$ is sub klt, and 
$D$ is nef, 
then $D$ is semi-ample. 
\end{cor}

\begin{proof}
If $\deg D>0$, 
then $D$ is ample. 
In particular, $D$ is semi-ample. 
From now on, we assume that $D$ is numerically trivial. 
By definition, $D=K_W+B_W+M_W$. 
Since $B_W$ is effective by Lemma \ref{g-lem5.1} and 
$M_W$ is nef by Theorem \ref{g-thm5.2}, 
$W$ is $\mathbb P^1$ or an elliptic curve. 
If $W=\mathbb P^1$, then $D\sim _{\mathbb Q}0$. 
Of course, $D$ is semi-ample. 
If $W$ is an elliptic curve, 
then $D\sim M_W$, that is, $D$ is linearly equivalent to $M_W$. 
In this case, $D$ is semi-ample by Theorem \ref{g-thm5.3}. 
Anyway, $D$ is always semi-ample. 
\end{proof}

Corollary \ref{g-cor5.5} is a key ingredient of the proof of 
the main theorem:~Theorem \ref{g-thm1.2}. 

\begin{cor}\label{g-cor5.5} 
If $\dim W=2$, $(V, B_V)$ is sub plt, $(W, B_W)$ is 
plt, 
and $D$ is nef and big, then 
$D$ is semi-ample. 
\end{cor}

\begin{proof}
Let $Z$ be an irreducible component of $\lfloor B_W\rfloor$. 
Then, by the definition of $B_W$ and Lemma 
\ref{g-lem4.1}, 
there exists an irreducible component $S$ of $B^{=1}_V$ such that 
$f(S)=Z$. 
Therefore, by Corollary \ref{g-cor4.2}, the natural map 
$\mathcal O_Z\to g_*\mathcal O_S(\lceil -B_S\rceil)$ is an isomorphism, 
where $K_S+B_S=(K_V+B_V)|_S$ and 
$g=f|_S$. 
Note that $(S, B_S)$ is sub klt and that 
$K_S+B_S\sim _{\mathbb Q} g^*(D|_Z)$. 
Thus, $D|_Z$ is semi-ample by Corollary \ref{g-cor5.4}. 
On the other hand, 
by Theorem \ref{g-thm5.2}, 
$M_W$ is nef since $M_W=D-(K_W+B_W)$ is 
$\mathbb Q$-Cartier and $W$ is a normal projective 
surface (see Lemma \ref{g-lem2.1}). 
Therefore, $2D-(K_W+B_W)=D+M_W$ is nef 
and big. 
Thus we obtain that $D$ is semi-ample by Lemma \ref{g-lem4.3}. 
\end{proof}

We close this section with a short remark on 
recent preprints \cite{fujino-slc-trivial} and 
\cite{fujino-fujisawa-liu}. 

\begin{rem}\label{g-rem5.6}
In \cite{fujino-slc-trivial}, the first author introduced the notion 
of basic slc-trivial fibrations, which is a generalization of 
that of lc-trivial fibrations, and got a much more general 
result than Theorem \ref{g-thm5.2} (see 
\cite[Theorem 1.2]{fujino-slc-trivial}). 
In \cite{fujino-fujisawa-liu}, we established the 
semi-ampleness of $M_W$ for basic slc-trivial fibrations 
when the base space $W$ is a curve (see \cite[Corollary 1.4]
{fujino-fujisawa-liu}). We strongly recommend the interested reader to 
see \cite{fujino-slc-trivial} and \cite{fujino-fujisawa-liu}. 
\end{rem}

\section{Minimal model program for surfaces}\label{g-sec6}

In this section, we quickly see a special case of the minimal model 
program for projective plt surfaces. 

We can easily check the following lemma by the usual 
minimal model program for surfaces. 
We recommend the interested reader to 
see \cite{fujino-surfaces} for the general theory 
of log surfaces. 

\begin{lem}\label{g-lem6.1} 
Let $(X, B)$ be a projective 
plt surface such that 
$B$ is a $\mathbb Q$-divisor and let $M$ be a nef $\mathbb Q$-divisor 
on $X$. 
Assume that $K_X+B+M$ is big. 
Then we can run the minimal model program with respect 
to $K_X+B+M$ and 
get a sequence of extremal contraction 
morphisms 
$$
(X, B+M)=:(X_0, B_0+M_0)\overset {\varphi_0}\longrightarrow \cdots 
\overset{\varphi_{k-1}}\longrightarrow (X_k, B_k+M_k)=:
(X^*, B^*+M^*)
$$ 
with the following properties: 
\begin{itemize}
\item[(i)] each $\varphi_i$ is a $(K_{X_i}+B_i+M_i)$-negative 
extremal birational contraction morphism, 
\item[(ii)] $K_{X_{i+1}}=\varphi_{i*}K_{X_i}$, 
$B_{i+1}=\varphi_{i*}B_i$, 
and $M_{i+1}=\varphi_{i*}M_i$ for every $i$, 
\item[(iii)] $M_i$ is nef for every $i$, and 
\item[(iv)] $K_{X^*}+B^*+M^*$ is nef 
and big. 
\end{itemize}
\end{lem}

\begin{proof}
If $K_{X_i}+B_i+M_i$ is not nef, then we can take 
an ample $\mathbb Q$-divisor $H_i$ and an effective 
$\mathbb Q$-divisor $\Delta_i$ on $X_i$ such that 
$K_{X_i}+B_i+M_i+H_i\sim _{\mathbb Q} K_{X_i}+\Delta_i$, 
$(X_i, \Delta_i)$ is plt, and $K_{X_i}+\Delta_i$ 
is not nef. 
By the cone and contraction 
theorem, 
we can construct a $(K_{X_i}+\Delta_i)$-negative 
extremal birational contraction morphism 
$\varphi_i: X_i\to X_{i+1}$. 
Since $H_i$ is ample, 
$\varphi_i$ is a $(K_{X_i}+B_i+M_i)$-negative 
extremal contraction morphism. 
Moreover, 
since $M_i$ is nef, 
$\varphi_i$ is $(K_{X_i}+B_i)$-negative. 
Therefore, $(X_{i+1}, B_{i+1})$ is plt by the negativity lemma. 
In particular, 
$X_{i+1}$ is $\mathbb Q$-factorial. 
By Lemma \ref{g-lem2.1}, 
we obtain that $M_{i+1}=\varphi_{i*}M_i$ is nef. 
Since $K_X+B+M$ is big by assumption, 
this minimal model program 
terminates. 
Then we finally get a model $(X^*, B^*+M^*)$ such that 
$K_{X^*}+B^*+M^*$ is nef and big. 
\end{proof}

If we put $\varphi:=\varphi_{k-1}\circ \cdots \circ \varphi_0: X\to 
X^*$, then we have 
$$
K_X+B+M=\varphi^*(K_{X^*}+B^*+M^*)+E
$$ 
for some effective $\varphi$-exceptional $\mathbb Q$-divisor 
$E$ on $X$ by the negativity lemma. 

\medskip 

We will use Lemma \ref{g-lem6.1} in the proof of 
the main theorem:~Theorem \ref{g-thm1.2}. 

\section{Proof of the main theorem:~Theorem \ref{g-thm1.2}}

In this section, let us prove the main theorem:~Theorem \ref{g-thm1.2}. 

\medskip 

Let $(X, \Delta)$ be a projective 
plt (resp.~lc) pair such that 
$\Delta$ is a $\mathbb Q$-divisor. 
Assume that $0<\kappa (X, K_X+\Delta)<\dim X$. 
Then we consider the Iitaka fibration 
$$f:=\Phi_{|m_0(K_X+\Delta)|}: X\dashrightarrow Y$$
where $m_0$ is a sufficiently large and divisible positive integer. 
By taking a suitable birational modification of $f: X\dashrightarrow Y$, we get 
a 
commutative diagram: 
$$
\xymatrix{
X \ar@{-->}[d]_-f& X'\ar[l]_-g\ar[d]^-{f'} \\ 
Y & \ar[l]^-hY'
}
$$ 
which satisfies the following properties: 
\begin{itemize}
\item[(i)] $X'$ and $Y'$ are smooth projective varieties, 
\item[(ii)] $g$ and $h$ are birational morphisms, and 
\item[(iii)] $K_{X'}+\Delta'=g^*(K_X+\Delta)$ such that 
$\Supp \Delta'$ is a simple normal crossing divisor on $X'$. 
\end{itemize}
We note that 
$(X', (\Delta')^{>0})$ is plt (resp.~lc) and that 
$$
H^0(X, \mathcal O_X(\lfloor m(K_X+\Delta)\rfloor))
\simeq H^0(X', \mathcal O_{X'}(\lfloor m(K_{X'}+(\Delta')^{>0})\rfloor))
$$
holds for every nonnegative integer $m$. 
Therefore, for the proof of the finite generation of 
the log canonical ring $R(X, \Delta)$, we may replace $(X, \Delta)$ with 
$(X', (\Delta')^{>0})$ and 
assume that the Iitaka fibration $f:X\dashrightarrow Y$ with respect to $K_X+\Delta$ 
is a morphism between smooth projective varieties. 
By construction, 
$\dim Y=\kappa (X, K_X+\Delta)$ and 
$\kappa (X_{\overline \eta}, K_{X_{\overline \eta}}+\Delta|_{X_{\overline \eta}})=0$, 
where $X_{\overline \eta}$ is the geometric generic fiber of $f:X\to Y$. 

By \cite[Theorem 2.1, Proposition 4.4, and 
Remark 4.5]{abramovich-karu}, 
we can construct a commutative diagram 
$$
\xymatrix{
U_{X'} \ar[d]\ar@{^{(}->}[r]& X'\ar[d]_-{f'}\ar[r]^-g & X\ar[d]^-f\\
U_{Y'} \ar@{^{(}->}[r]& Y'\ar[r]_-h& Y
}
$$
such that $g$ and $h$ are projective 
birational morphisms, $X'$ and $Y'$ are normal projective  
varieties, the inclusions $U_{X'}\subset 
X'$ and $U_{Y'}\subset Y'$ are toroidal embeddings without 
self-intersection satisfying the 
following conditions: 
\begin{itemize}
\item[(a)] $f'$ is toroidal with respect to $(U_{X'}\subset X')$ and 
$(U_{Y'}\subset Y')$,  
\item[(b)] $f'$ is equidimensional, 
\item[(c)] $Y'$ is smooth, 
\item[(d)] $X'$ has only quotient singularities, and 
\item[(e)] there exists a dense Zariski open set $U$ of $X$ such 
that $g$ is an isomorphism 
over $U$, $U_{X'}=g^{-1}(U)$, and $U\cap \Delta=\emptyset$. 
\end{itemize} 
Although it is not treated explicitly in \cite{abramovich-karu}, 
we can make $g:X'\to X$ satisfy condition (e) 
(see Remark \ref{g-rem7.1}). 

\begin{rem}\label{g-rem7.1}
In this remark, we will freely use the 
notation in \cite{abramovich-karu}. 
For condition (e), 
it is sufficient to prove that there exists a 
Zariski open set $U$ of $X$ 
such that $U_{X'}=m^{-1}_X(U)$ and that 
$m_X$ is an isomorphism over $U_{X'}$ 
in \cite[Theorem 2.1]{abramovich-karu}. 
Precisely speaking, we enlarge $Z$ and may assume that 
$X\setminus Z$ is a Zariski open set of the original 
$X$ in \cite[2.3]{abramovich-karu}, and 
enlarge $\Delta$ suitably in \cite[2.5]{abramovich-karu}. 
Then we can construct $m_X: X'\to X$ such that 
$U$ is a Zariski open set of $X$, 
$U_{X'}=m^{-1}_X(U)$, and $m_X: U_{X'}\to U$ is an isomorphism. 
\end{rem}
 
By condition (e), we have $\Supp \Delta'\subset X'\setminus U_{X'}$, where 
$\Delta'$ is a $\mathbb Q$-divisor defined by $K_{X'}+\Delta'=g^*(K_X+\Delta)$. 
We note that $(X', (\Delta')^{>0})$ is plt (resp.~lc) and that 
$$
H^0(X, \mathcal O_X(\lfloor m(K_X+\Delta)\rfloor))
\simeq H^0(X', \mathcal O_{X'}(\lfloor m(K_{X'}+(\Delta')^{>0})\rfloor))
$$
holds for every nonnegative integer $m$. 
Therefore, by replacing $f:X\to Y$ and $(X, \Delta)$ with $f': X'\to Y'$ 
and $(X', (\Delta')^{>0})$, respectively, 
we may assume that $f:X\to Y$ satisfies conditions (a), (b), (c), (d), and $\Supp 
\Delta\subset X\setminus U_X$, where 
$(U_X\subset X)$ is the toroidal structure in (a).  

Since $\kappa (X, K_X+\Delta)>0$, we can 
take a divisible positive integer $a$ 
such that 
$$H^0(X, \mathcal O_X(a(K_X+\Delta)))\ne 0. 
$$  
Therefore, there exists an effective Cartier divisor $L$ on $X$ such that 
$$
a(K_X+\Delta)\sim L.  
$$
We put 
$$
G:=\max \{ N \,|\, \text{$N$ is an effective $\mathbb Q$-divisor 
on $Y$ such that $L\geq f^*N$}\}.  
$$ 
Then we set 
$$
D:=\frac{1}{a}G \quad \text{and}\quad F:=\frac{1}{a}(L-f^*G). 
$$ 
By definition, we have 
$$
K_X+\Delta\sim _{\mathbb Q} f^*D+F. 
$$
\begin{lem}\label{g-lem7.2}
For every nonnegative integer $i$, 
the natural map 
$$
\mathcal O_Y\to f_*\mathcal O_X(\lfloor iF\rfloor)
$$ 
is an isomorphism. 
\end{lem}
\begin{proof}
By definition, $F$ is effective. 
Therefore, we have a natural nontrivial map 
$$\mathcal O_Y\to f_*\mathcal O_X(\lfloor iF\rfloor)
$$ for every nonnegative integer $i$. 
By $\kappa (X_{\overline \eta}, K_{X_{\overline \eta}}+\Delta|_{X_{\overline 
\eta}})=0$, 
we have $\kappa (X_{\overline \eta}, F|_{X_{\overline \eta}})=0$. 
Thus, we see that 
$f_*\mathcal O_X(\lfloor iF\rfloor)$ is a reflexive sheaf of 
rank one since $f$ is equidimensional. 
Moreover, since $Y$ is smooth, $f_*\mathcal O_X(\lfloor iF\rfloor)$ is 
invertible. Let $P$ be any prime divisor on $Y$. 
By construction, $\Supp F$ does not contain the whole fiber 
of $f$ over the generic point of $P$. 
Therefore, we obtain that 
$\mathcal O_Y\to f_*\mathcal O_X(\lfloor iF\rfloor)$ is 
an isomorphism in codimension one. 
This implies that the natural map 
$$\mathcal O_Y\to f_*\mathcal O_X(\lfloor iF\rfloor)$$ is 
an isomorphism for every nonnegative integer $i$. 
\end{proof} 
By construction and Lemma \ref{g-lem7.2}, 
there exists a divisible positive integer $r$ such that 
$r(K_X+\Delta)$ and $rD$ are Cartier and that 
$$
H^0(X, \mathcal O_X(mr(K_X+\Delta)))\simeq H^0(Y, \mathcal O_Y(mrD))
$$ 
holds for every nonnegative integer $m$. 
In particular, $D$ is a big $\mathbb Q$-divisor on $Y$. 
We put $B:=\Delta-F$ and consider $K_X+B\sim _{\mathbb Q}f^*D$. 
Let $p: V\to X$ be a birational morphism from a smooth projective 
variety $V$ such that 
$K_V+B_V=p^*(K_X+B)$ and that 
$\Supp B_V$ is a simple normal crossing divisor.  
$$
\xymatrix{
V\ar[r]^-p \ar[dr]_-\pi& X\ar[d]^-f \\ 
& Y
}
$$ 
It is obvious that $K_V+B_V\sim _{\mathbb Q}\pi^*D$ holds. 
Since $p_*\mathcal O_V(\lceil -(B^{<1}_V)\rceil)\subset 
\mathcal O_X(kF)$ for some divisible positive integer $k$, 
the natural map 
$\mathcal O_Y\to \pi_*\mathcal O_V(\lceil -(B^{<1}_V)\rceil)$ is an isomorphism. 
For any prime divisor $P$ on $Y$, we put 
$$
b_P:=\max \left\{t \in \mathbb Q\, |\, 
 {\text{$(X, B+tf^*P)$ is sub lc over the generic point of $P$}} \right\}.  
$$ 
Then we set  
$$
B_Y:=\sum _P (1-b_P)P 
$$ 
as in Section \ref{g-sec5}. Since $K_V+B_V=p^*(K_X+B)$ and 
the natural map $\mathcal O_Y\to \pi_*\mathcal O_V(\lceil -(B^{<1}_V)\rceil)$ is 
an isomorphism, $B_Y$ is the discriminant $\mathbb Q$-divisor 
of $\pi: (V, B_V)\to Y$ and is a boundary $\mathbb Q$-divisor 
on $Y$ 
(see Lemma \ref{g-lem5.1}). 
By construction, we have $b_P=1$ if $P\cap U_Y\ne \emptyset$, 
where $(U_Y\subset Y)$ is the toroidal structure in (a). 
Therefore, $\Supp B_Y\subset Y\setminus U_Y$. 

\medskip 

From now on, we assume that $(X, \Delta)$ is plt and $\kappa(X, K_X+\Delta)=2$. 
Then $(V, B_V)$ is sub plt and $Y$ is a smooth projective 
surface. 
As in Section \ref{g-sec5}, 
we write 
$$
D=K_Y+B_Y+M_Y,  
$$ 
where $M_Y$ is the moduli $\mathbb Q$-divisor 
of $K_V+B_V\sim _{\mathbb Q}\pi^*D$.  
As we saw above, 
$\Supp B_Y\subset Y\setminus U_Y$, where $(U_Y\subset Y)$ is the 
toroidal structure in (a). 
In particular, this means that $\Supp B_Y$ is a 
simple normal crossing divisor 
on $Y$ because $Y$ is smooth. 
By Lemma \ref{g-lem4.1}, 
$\lfloor B_Y\rfloor$ is a disjoint union of 
some smooth prime divisors. 
Therefore, $(Y, B_Y)$ is plt. 
By Lemma \ref{g-lem6.1}, there exists a projective 
birational contraction morphism 
$\varphi:Y\to Y'$ onto a normal projective 
surface $Y'$ such that 
$D'=K_{Y'}+B_{Y'}+M_{Y'}$ is nef and big and that 
$D=\varphi^*D'+E$ for some 
effective $\varphi$-exceptional $\mathbb Q$-divisor $E$ on $Y$. 
Of course, $D'$, $K_{Y'}$, $B_{Y'}$, and $M_{Y'}$ are 
the pushforwards of $D$, $K_Y$, $B_Y$, and $M_Y$ by $\varphi$, respectively. 
$$
\xymatrix{
V \ar[d]_-\pi\ar[dr]^-{\pi'}& \\
Y \ar[r]_-\varphi& Y'
}
$$ 
By replacing $V$ birationally, we may further assume 
that the union of $\Supp B_V$ and $\Supp \pi^*E$ is a simple 
normal crossing divisor on $V$. 
We consider 
$$
K_V+B_V-\pi^*E\sim _{\mathbb Q} \pi'^*D'. 
$$ 
We note that the natural map 
$$
\mathcal O_{Y'}\to \pi'_*\mathcal O_V(\lceil -(B_V-\pi^*E)^{<1}\rceil)
$$
is an isomorphism since $\pi_*\mathcal O_V(\lceil -(B_V-\pi^*E)^{<1}\rceil)
\subset \mathcal O_Y(kE)$ for some divisible positive integer $k$ 
and $\mathcal O_{Y'}\overset{\sim}\longrightarrow \varphi_*\mathcal O_Y(kE)$. 
By construction, $(Y', B_{Y'})$ is plt (see Lemma \ref{g-lem6.1}) and 
$B_{Y'}$ is the discriminant $\mathbb Q$-divisor 
of $\pi': (V, B_V-\pi^*E)\to Y'$. 
Therefore, by Corollary \ref{g-cor5.5}, 
$D'$ is semi-ample. 
Thus, we obtain that 
$$
\bigoplus _{m\geq 0} H^0(Y, \mathcal O_Y(\lfloor mD\rfloor))
\simeq 
\bigoplus_{m\geq 0} H^0(Y', \mathcal O_{Y'}(\lfloor mD'\rfloor))
$$ 
is a finitely generated $\mathbb C$-algebra. 
This implies that the log canonical 
ring $R(X, \Delta)$ of $(X, \Delta)$ is also a finitely generated $\mathbb C$-algebra. 

Hence we have finished the proof of Theorem \ref{g-thm1.2}. 
%%%%%%%%%%%%%%%

\end{document}